\documentclass[12pt,a4]{amsart}
\usepackage[a4paper, left=26mm, right=26mm, top=28mm, bottom=34mm]{geometry}

\usepackage[all]{xy}
\usepackage{amsmath}
\usepackage{amscd}
\usepackage{amssymb}
\usepackage{amsthm}
\usepackage{url}
\usepackage{mathrsfs}
\usepackage{comment}

\newtheorem{thm}{Theorem}[section]
\newtheorem{lem}[thm]{Lemma}
\newtheorem{cor}[thm]{Corollary}
\newtheorem{prop}[thm]{Proposition}

\theoremstyle{definition}

\newtheorem{rem}[thm]{Remark}
\newtheorem{defn}[thm]{Definition}

\newtheorem{ex}[thm]{Example}

\newcommand{\adj}[4]{#1\negmedspace: #2\rightleftarrows #3:\negmedspace #4}

\def\F{{\mathbb F}}

\def\Q{{\mathbb Q}}

\def\Z{{\mathbb Z}}

\def\O{{\mathcal O}}

\def\L{{\mathbb L}}

\def\Hom{\mathop{\mathrm{Hom}}\nolimits}

\def\Ker{\mathop{\mathrm{Ker}}\nolimits}
\def\id{\mathop{\mathrm{id}}\nolimits}

\def\Spec{\mathop{\rm Spec}}

\def\p{\mathfrak{p}}

\def\Catinfty{\mathop{\mathrm{Cat}_{\infty}}\nolimits}
\def\Perf{\mathop{\mathrm{Perf}}\nolimits}
\def\Mod{\mathop{\mathrm{Mod}}\nolimits}

\def\Ring{\mathop{\mathrm{Ring}}\nolimits}
\def\E{\mathop{\mathbb{E}}\nolimits}
\def\CAlg{\mathop{\mathrm{CAlg(Sp)}}\nolimits}
\def\Map{\mathop{\mathrm{Map}}\nolimits}
\def\colim{\mathop{\mathrm{colim}}\nolimits}
\def\arc{\mathop{\mathrm{arc}}\nolimits}
\def\DD{\mathop{\mathrm{DD}}\nolimits}
\def\Vect{\mathop{\mathrm{Vect}}\nolimits}

\numberwithin{equation}{section}

\usepackage[T1]{fontenc}

\begin{document}

\title[$p$-complete arc-descent for perfect complexes]{$p$-complete arc-descent for perfect complexes over integral perfectoid rings}

\author{Kazuhiro Ito}
\address{Kavli Institute for the Physics and Mathematics of the Universe (WPI), The University of Tokyo,
5-1-5 Kashiwanoha, Kashiwa, Chiba, 277-8583, Japan}
\email{kazuhiro.ito@ipmu.jp}

% \date{\today}


%14G45 Perfectoid spaces and mixed characteristic
% 11G25  Varieties over finite and local fields
% 11F85  $p$-adic theory, local fields 
% 11G15  Complex multiplication and moduli of abelian varieties 
% 11G18  Arithmetic aspects of modular and Shimura varieties
% 14F20  	\'Etale and other Grothendieck topologies and (co)homologies
% 14G15  Finite ground fields
% 14G20  Local ground fields
% 14C25  Algebraic cycles
% 14G35  Modular and Shimura varieties
% 14J28  $K3$ surfaces and Enriques surfaces
% 14K22  Complex multiplication
%14L05 Formal groups, p-divisible groups

\maketitle

\begin{abstract}
We prove $p$-complete arc-descent results for finite projective modules and perfect complexes over integral perfectoid rings.
Using our results,
we clarify a reduction argument in the proof of the classification of $p$-divisible groups over integral perfectoid rings given by Scholze--Weinstein.
\end{abstract}

\section{Introduction} \label{Section:Introduction}

We fix a prime number $p >0$.
In this paper, we will prove $p$-complete $\arc$-descent results for finite projective modules and perfect complexes over
($p$-complete integral) perfectoid rings.

Before stating our main results, we first recall the following descent result due to Bhatt--Scholze \cite[Theorem 11.2]{BS} for the $v$-topology and Bhatt--Mathew \cite[Theorem 5.16]{BM} for the $\arc$-topology.
Following their work, we will use the $\infty$-categorical language in this paper.
Let $\Catinfty$ denote the $\infty$-category of (small) $\infty$-categories.
For a commutative ring $A$, the $\infty$-category of perfect complexes over $A$ is denoted by $\Perf(A)$.
Recall that, for a homomorphism $A \to B$ of commutative rings, we have a base change functor
$\Perf(A) \to \Perf(B)$, $K \mapsto K \otimes^\L_A B$.

\begin{thm}[{Bhatt--Scholze, Bhatt--Mathew}]\label{Theorem:descent for pefect rings}
The functor
$A \mapsto \Perf(A)$
from the category of perfect $\F_p$-algebras to $\Catinfty$ satisfies $\arc$-hyperdescent.
\end{thm}

The main result of this paper is the following analogous statement for perfectoid rings (in the sense of \cite[Definition 3.5]{BMS}).
Here we use the $\varpi$-complete $\arc$-topology introduced in \cite[Section 2.2.1]{CS}.

\begin{thm}\label{Theorem:descent for perfectoid}
Let $R$ be a perfectoid ring and $\varpi \in R$ an element with $p \in (\varpi^p)$ such that $R$ is $\varpi$-complete.
Then the functors
\[
A \mapsto \Perf(A^\flat), \quad A \mapsto \Perf(W(A^\flat)), \quad and \quad A \mapsto \Perf(A)
\]
from the category of $\varpi$-complete perfectoid $R$-algebras to $\Catinfty$ satisfy $\varpi$-complete $\arc$-hyperdescent.
Here $A^\flat:=\varprojlim_{x \mapsto x^p} A/pA$ is the tilt of $A$ and $W(A^\flat)$ is the ring of Witt vectors of $A^\flat$.
\end{thm}

We will make precise what we mean by ``$\varpi$-complete $\arc$-hyperdescent'' in Definition \ref{Definition:hyperdescent}.
From this theorem,
we can deduce $\varpi$-complete $\arc$-descent results for finite projective modules over perfectoid rings, which can be stated using classical category theory; see Corollary \ref{Corollary:finite projective modules}.

We need an analogue of Theorem \ref{Theorem:descent for pefect rings} for derived quotients of perfect rings to prove Theorem \ref{Theorem:descent for perfectoid}.
We formulate it in terms of $\mathbb{E}_\infty$-rings and their modules; see Theorem \ref{Theorem:descent for perfect complex for derived quotients}.
We will present two proofs of Theorem \ref{Theorem:descent for perfect complex for derived quotients}; it can be proved in the same way as Theorem \ref{Theorem:descent for pefect rings}, and also can be deduced from Theorem \ref{Theorem:descent for pefect rings}.
Even if one is only interested in descent results for finite projective modules (Corollary \ref{Corollary:finite projective modules}), such an analogue will be essential.

\begin{rem}\label{Remark:Henkel's thesis Intro}
A similar statement to Theorem \ref{Theorem:descent for perfectoid} has been conjectured by Henkel in \cite[Conjecture A]{Henkel}.
In fact, the proof of Theorem \ref{Theorem:descent for perfectoid} shows that \cite[Conjecture A]{Henkel} is true; see Remark \ref{Remark:Henkel's thesis} for additional details.
\end{rem}

As an application of Theorem \ref{Theorem:descent for perfectoid} (or Corollary \ref{Corollary:finite projective modules}), we will discuss the following classification result for $p$-divisible groups over perfectoid rings obtained by Lau \cite[Theorem 9.8]{Lau2018} in the case where $p \geq 3$, and by Scholze--Weinstein \cite[Theorem 17.5.2]{Scholze-Weinstein} in general; see Theorem \ref{Theorem:classification over perfectoid ring} for a more precise statement.

\begin{thm}[{Lau, Scholze--Weinstein}]\label{Theorem:Classifition p-divisible group Intro}
Let $A$ be a perfectoid ring.
The category of $p$-divisible groups over $A$
is anti-equivalent to the category of minuscule Breuil--Kisin--Fargues modules for $A$.
\end{thm}

\begin{rem}\label{Remark:Scholze-Weinstein gap}
The strategy of Scholze--Weinstein is to deduce the general statement from the classification of $p$-divisible groups over perfectoid valuation rings of rank $\leq 1$ with algebraically closed fraction fields, which is proved by Berthelot \cite[Corollaire 3.4.3]{Berthelot} in the equal characteristic case, and by Scholze--Weinstein \cite[Theorem 14.4.1]{Scholze-Weinstein} (based on \cite[Theorem B]{Scholze-Weinstein2013}) in the mixed characteristic case.
In their original proof, however, there seems to be a technical issue in this reduction procedure.\footnote{More precisely, we should replace the $v$-cover $S$ over $R^\flat$ given in the last paragraph of the proof of \cite[Theorem 17.5.2]{Scholze-Weinstein} with its $\xi_0$-completion to conclude that $W(S)/(\xi)$ is a perfectoid ring over $R$ whose tilt is isomorphic to $S$ (see also Proposition \ref{Proposition:tilting and arc cover}). Here $\xi=(\xi_0, \xi_1, \dotsc) \in W(R^\flat)$ is a generator of the kernel of the usual map $\theta \colon W(R^\flat) \to R$.
Accordingly, we need to use Corollary \ref{Corollary:finite projective modules} (or its ``$\varpi$-complete $v$-descent'' analogue, which, to the best of our knowledge, has not been proved before in the literature, either) instead of Theorem \ref{Theorem:descent for pefect rings} or \cite[Theorem 4.1]{BS}. We thank P.\ Scholze for e-mail correspondence on the proof of \cite[Theorem 17.5.2]{Scholze-Weinstein}.}
We will verify it by using Theorem \ref{Theorem:descent for perfectoid} (or Corollary \ref{Corollary:finite projective modules}) in Section \ref{Section:The classification of p-divisible groups over perfectoid rings}.
\end{rem}

This paper is organized as follows.
In Section \ref{Section:Preliminaries}, we recall the definitions and some basic properties of perfectoid rings and $\varpi$-complete $\arc$-covers.
We also collect some results from the theory of $\E_\infty$-rings and their modules, which are used in the proof of Theorem \ref{Theorem:descent for perfectoid}.
In Section \ref{Section:An analogue for derived quotients of perfect rings}, we state and prove an analogue of Theorem \ref{Theorem:descent for pefect rings} for derived quotients of perfect rings (Theorem \ref{Theorem:descent for perfect complex for derived quotients}).
In Section \ref{Section:Proofs of main results}, we prove Theorem \ref{Theorem:descent for perfectoid} and deduce descent results for finite projective modules (Corollary \ref{Corollary:finite projective modules}) from it.
In Section \ref{Section:The classification of p-divisible groups over perfectoid rings}, we give a proof of Theorem \ref{Theorem:Classifition p-divisible group Intro}, following the approach of Scholze--Weinstein.

\section{Preliminaries}\label{Section:Preliminaries}

\subsection{Perfectoid rings and $\varpi$-complete $\arc$-covers}\label{Subsection:Perfectoid rings and varpi-complete arc-covers}

In this subsection, we recall some basic facts about perfectoid rings and $\varpi$-complete $\arc$-covers.
Our basic references are \cite[Section 3]{BMS} and \cite[Section 2]{CS}.

Let us first recall the notion of a $\varpi$-complete $\arc$-cover from \cite[Section 2.2.1]{CS}.
Let $R$ be a commutative ring and $\varpi \in R$ an element.
We say that a homomorphism $A \to B$ of $R$-algebras is a \textit{$\varpi$-complete $\arc$-cover} if, for any homomorphism $A \to V$ with $V$ a $\varpi$-complete valuation ring of rank $\leq 1$, there exist an extension of $V \hookrightarrow W$ of $\varpi$-complete valuation rings of rank $\leq 1$ and a homomorphism $B \to W$ such that the following diagram commutes:
\[
\xymatrix{
A \ar^-{}[r]  \ar[d]_-{} & B \ar[d]_-{}  \\
V \ar[r]^-{} & W.
}
\]
A $0$-complete $\arc$-cover is just an $\arc$-cover in the sense of \cite[Definition 1.2]{BM}.
The reduction modulo $\varpi$ of a $\varpi$-complete $\arc$-cover is an $\arc$-cover.

Let $R$ be a \textit{perfectoid ring} in the sense of \cite[Definition 3.5]{BMS}, i.e.\ $R$ is $\varpi$-complete for some element $\varpi \in R$ with $p \in (\varpi^p)$, the Frobenius map $R/pR \to R/pR$ is surjective, and the kernel of $\theta \colon W(R^\flat) \to R$ is principal.
Here
\[
R^\flat:=\varprojlim_{x \mapsto x^p} R/pR
\]
is the tilt of $R$ and
$\theta \colon W(R^\flat) \to R$ is the unique ring homomorphism whose reduction modulo $p$ is the projection map
$R^\flat \to R/pR$, $(x_0, x_1, \dotsc) \mapsto x_0$.
We note that $R$ is $p$-complete.

Let $\varpi \in R$ be an element as above.
Let $\mathcal{C}_{R, \varpi}$ denote the category of $\varpi$-complete perfectoid $R$-algebras.
Every diagram $B \leftarrow A \rightarrow B'$ in $\mathcal{C}_{R, \varpi}$ has a colimit, which is given by the $\varpi$-completion of $B \otimes_A B'$; see, for instance, \cite[Proposition 2.1.11]{CS}.
Hence we can define, in the usual way, the \v{C}ech conerve $A \to B^{\bullet}$ for a homomorphism $A \to B$ in $\mathcal{C}_{R, \varpi}$, which is an augmented cosimplicial object in $\mathcal{C}_{R, \varpi}$, and the notion of a \textit{$\varpi$-complete $\arc$-hypercover} $A \to B^{\bullet}$ using $\varpi$-complete $\arc$-covers.

In this paper, we will check that some functors on $\mathcal{C}_{R, \varpi}$ (e.g.\ $A \mapsto \Perf(A)$) satisfy $\varpi$-complete $\arc$-(hyper)descent.
For the sake of clarity, let us recall what it means.

\begin{defn}\label{Definition:hyperdescent}
Let $\mathcal{F} \colon \mathcal{C}_{R, \varpi} \to \mathcal{D}$ be a (covariant) functor to an $\infty$-category $\mathcal{D}$.
We say that the functor $\mathcal{F}$ satisfies \textit{$\varpi$-complete $\arc$-descent}
(resp.\ \textit{$\varpi$-complete $\arc$-hyperdescent})
if $\mathcal{F}$ preserves finite products and if for every $\varpi$-complete $\arc$-cover $A \to B$ in $\mathcal{C}_{R, \varpi}$ with \v{C}ech conerve $A \to B^{\bullet}$
(resp.\ every $\varpi$-complete $\arc$-hypercover $A \to B^{\bullet}$)
we have
\[
\mathcal{F}(A) \overset{\sim}{\to} \lim_{\Delta} \mathcal{F}(B^\bullet) \quad \text{in} \quad \mathcal{D},
\]
i.e.\ $\mathcal{F}(A)$ is a limit of the cosimplicial diagram $\mathcal{F}(B^\bullet) \colon \Delta \to \mathcal{D}$.
Here $\Delta$ denotes the simplex category.
(See also \cite[Section A.3.3 and Section A.5.7]{SAG}.)
\end{defn}

\begin{ex}\label{Example:pefect rings case}
\begin{enumerate}
    \item An $\F_p$-algebra is perfectoid if and only if it is perfect; see \cite[Example 3.15]{BMS}. If $R$ is a perfect $\F_p$-algebra, then $\varpi$ can be any element such that $R$ is $\varpi$-complete. If $\varpi=0$, then $\mathcal{C}_{R, 0}$ is just the category of perfect $R$-algebras.
    \item For every perfectoid algebra $R$, there is an element $\varpi \in R$ such that $\varpi^p$ is a unit multiple of $p$; see \cite[Lemma 3.9]{BMS}.
    For such an element $\varpi$, we see that $\mathcal{C}_{R, \varpi}$ is the category of perfectoid $R$-algebras and a $\varpi$-complete $\arc$-cover in $\mathcal{C}_{R, \varpi}$ is simply a $p$-complete $\arc$-cover.
\end{enumerate}
\end{ex}

We set $x^\sharp:=\theta([x])$ for an element $x \in R^\flat$, where $[-]$ denotes the Teichm\"uller lift.
By \cite[Lemma 3.9]{BMS}, there is an element $\varpi^\flat \in R^\flat$ such that $(\varpi^\flat)^\sharp$ is a unit multiple of $\varpi$.
The perfect ring $R^\flat$ is $\varpi^\flat$-complete.

\begin{prop}[{\cite[Proposition 2.1.9]{CS}}]\label{Proposition:tilting and arc cover}
The functor $A \mapsto A^\flat$ induces an equivalence from the category $\mathcal{C}_{R, \varpi}$ of $\varpi$-complete perfectoid $R$-algebras to the category $\mathcal{C}_{R^\flat, \varpi^\flat}$ of $\varpi^\flat$-complete perfect $R^\flat$-algebras.
The inverse functor is given by $B \mapsto W(B)/(\xi)$, where $\xi$ is a generator of the kernel of
$\theta \colon W(R^\flat) \to R$.
Moreover, a map $A \to B$ in $\mathcal{C}_{R, \varpi}$ is a $\varpi$-complete $\arc$-cover if and only if the induced map $A^\flat \to B^\flat$ is a $\varpi^\flat$-complete $\arc$-cover.
\end{prop}

\begin{proof}
This proposition follows from \cite[Proposition 2.1.9]{CS}.
(See also \cite[Lemma 2.2.2]{CS}.)
\end{proof}

\subsection{Preliminaries on $\E_\infty$-rings and their modules}\label{Preliminaries on Einfty-rings and their modules}

In this subsection, we review some terminology and results 
from the theory of $\E_\infty$-rings and their modules.
Our basic references are \cite{HA} and \cite{SAG}.

Let $\CAlg$ denote the $\infty$-category of $\E_\infty$-rings; see \cite[Section 7.1]{HA} for the definition and results which we state below.
We say that an $\E_\infty$-ring $A$ is \textit{connective} (resp.\ \textit{discrete}) if for $n < 0$ (resp.\ for $n \neq 0$) the homotopy group $\pi_n(A)$ is trivial.
Let $\CAlg^\mathrm{cn} \subset \CAlg$ be the full subcategory spanned by the connective $\E_\infty$-rings.
We have a natural fully faithful functor
\[
\Ring \to \CAlg^\mathrm{cn}
\]
from the category $\Ring$ of commutative rings, whose essential image is the full subcategory spanned by the discrete $\E_\infty$-rings.
This functor admits a left adjoint which sends a connective $\E_\infty$-ring $A$ to $\pi_0(A)$.
We will regard a commutative ring as an $\E_\infty$-ring via this functor.

\begin{ex}\label{Example:derived quotient}
An example of an $\E_\infty$-ring is a derived quotient of a commutative ring, that is defined as follows.
We first remark that the $\infty$-category $\CAlg$ admits colimits and limits.
In particular, a diagram $B \leftarrow A \rightarrow B'$ in $\CAlg$ admits a colimit, denoted by
$B \otimes^\L_A B'$.
Let $R$ be a commutative ring and let $x_1, \dotsc, x_r \in R$ be elements.
Then we define the derived quotient of $R$ by elements $x_1, \dotsc, x_r$ as
\[
R/^\L (x_1, \dotsc, x_r) := R \otimes^\L_{\Z[X_1, \dotsc, X_r]} \Z,
\]
where $\Z[X_1, \dotsc, X_r] \to R$ is defined by $X_i \mapsto x_i$ and $\Z[X_1, \dotsc, X_r] \to \Z$ is defined by $X_i \mapsto 0$.
The $\E_\infty$-ring $R/^\L (x_1, \dotsc, x_r)$ is connective.
\end{ex}

For an $\E_\infty$-ring $A$,
let $\Mod(A)$ denote the $\infty$-category of module spectra over $A$; see \cite[Notation 7.1.1.1]{HA} for the definition.
We simply say ``module'' instead of ``module spectrum'' when there is no ambiguity.
The $\infty$-category $\Mod(A)$ is stable and has the structure of a symmetric monoidal $\infty$-category.
Let us write $\otimes^{\L}_{A}$ for the tensor product.

\begin{defn}[{\cite[Definition 7.2.4.1]{HA}}]\label{Definition:perfect modules}
Let
\[
\Perf(A) \subset \Mod(A)
\]
denote the smallest stable subcategory of $\Mod(A)$ which contains $A$ and is closed under retracts.
We say that a module $K \in \Mod(A)$ is \textit{perfect} if $K \in \Perf(A)$.
\end{defn}

\begin{ex}\label{Example:derived infinity category}
Let $A$ be a discrete $\E_\infty$-ring.
By \cite[Theorem 7.1.2.13]{HA}, we have a natural equivalence
\[
\Mod(A) \cong \mathcal{D}(\pi_0(A)).
\]
Here $\mathcal{D}(\pi_0(A))$ is the derived $\infty$-category of $\pi_0(A)$.
Via this equivalence, we may identify $\Perf(A)$ with
the $\infty$-category of perfect complexes over $\pi_0(A)$.
This follows, for instance, from the fact that (for every $\E_\infty$-ring $A$) a module $K \in \Mod(A)$ is perfect if and only if $K$ is a compact object of $\Mod(A)$ in the sense of \cite[Definition 5.3.4.5]{HTT}; see \cite[Proposition 7.2.4.2]{HA}.
(Compare \cite[Tag 07LT]{SP}.)
\end{ex}

For a map $A \to B$ of $\E_\infty$-rings,
we have a forgetful functor
$\Mod(B) \to \Mod(A)$.
This admits a left adjoint
$- \otimes^\L_A B \colon \Mod(A) \to \Mod(B)$, which induces a functor $\Perf(A) \to \Perf(B)$.

We will need the following notion of a perfect module with Tor-amplitude in $[a, b]$.
Here we use homological indexing conventions.

\begin{defn}\label{Definition:Tor-amplitude}
Let $A$ be a connective $\E_\infty$-ring (and therefore we have a natural map $A \to \pi_0(A)$).
Let $K \in \Perf(A)$ and let $a, b$ be integers with $a \leq b$.
We say that $K$ has
\textit{Tor-amplitude in} $[a, b]$
if the base change
$K \otimes^\L_A \pi_0(A)$ has Tor-amplitude in $[a, b]$ (or in other words, for every discrete $\pi_0(A)$-module $M$, we have $\pi_n(K \otimes^\L_A M)=0$ for every $n \notin [a, b]$).
Let
\[
\Perf_{[a, b]}(A) \subset \Perf(A)
\]
be the full subcategory spanned by the perfect modules with Tor-amplitude in $[a, b]$.
\end{defn}

\begin{lem}\label{Lemma:Tor-amplitude}
Let $A$ be a connective $\E_\infty$-ring.
Let $a, b$ be integers with $a \leq b$.
\begin{enumerate}
    \item Assume that there is an integer $r \geq 0$ such that $\pi_m(A)=0$ for $m \geq r$.
    We put $n:=b-a+r$.
    Then the $\infty$-category $\Perf_{[a, b]}(A)$ is
    equivalent to an $n$-category (in the sense of \cite[Definition 2.3.4.1]{HTT}).
    \item Let $A \to B$ be a map of connective $\E_\infty$-rings.
    Assume that the induced map $\Spec \pi_0(B) \to \Spec \pi_0(A)$ is surjective.
    Then a perfect module $K \in \Perf(A)$ has Tor-amplitude in $[a, b]$ if and only if $K\otimes^\L_A B \in \Perf(B)$ has Tor-amplitude in $[a, b]$.
\end{enumerate}
\end{lem}

\begin{proof}
(1) By \cite[Proposition 2.3.4.18]{HTT}, it suffices to prove that the mapping space $\Map(K, L)$ is $(n-1)$-truncated for all $K, L \in \Perf_{[a, b]}(A)$.
By \cite[Proposition 7.2.4.4]{HA}, there is a perfect module $K^{\vee} \in \Perf(A)$ such that the underlying space of $K^\vee \otimes^{\L}_A L$ is equivalent to $\Map(K, L)$ functorially in $L \in \Mod(A)$.
(The module $K^\vee$ is a dual of $K$ in the symmetric monoidal $\infty$-category $\Mod(A)$; see also the proof of \cite[Proposition 2.7.28]{DAGVIII}.)
We have
\[
K^\vee \otimes^{\L}_A \pi_0(A) \cong R\Hom_{\pi_0(A)}(K \otimes^{\L}_A \pi_0(A), \pi_0(A)).
\]
It follows that $K^\vee \in \Perf_{[-b, -a]}(A)$, and hence $K^\vee \otimes^{\L}_A L \in \Perf_{[-b+a, b-a]}(A)$ for $L \in \Perf_{[a, b]}(A)$.
Then we see that the homotopy group $\pi_m(K^\vee \otimes^{\L}_A L)$ is trivial for $m \geq n$ by \cite[Proposition 7.2.4.23 (5)]{HA} and the assumption on $\pi_m(A)$.
Thus $\Map(K, L)$ is $(n-1)$-truncated as desired.

(2) This follows from the following fact.
Let $K \in \Perf(R)$ be a perfect complex over a commutative ring $R$.
If $K \otimes^\L_R \kappa(x) \in \Perf_{[a, b]}(\kappa(x))$ for every point $x \in \Spec R$, where $\kappa(x)$ is the residue field of $x$, then we have $K \in \Perf_{[a, b]}(R)$; see the proof of \cite[Theorem 11.2 (2)]{BS}.
\end{proof}

We record some results on perfect modules over ``derived complete'' $\E_\infty$-rings.
Let $R$ be a commutative ring and let $I=(x_1, \dotsc, x_r) \subset R$ be a finitely generated ideal.
Let $R \rightarrow A$ be a map of $\E_\infty$-rings.
We say that a module $M \in \Mod(A)$ is \textit{derived $I$-complete} if $M$ is derived $I$-complete (in the sense of \cite[Tag 091S]{SP}) when we regard it as an object of $\mathcal{D}(R)$; see also \cite[Definition 7.3.1.1 and Corollary 7.3.3.6]{SAG}.

\begin{prop}\label{Proposition:derived complete and reduction modulo I}
Let $A$ be a connective $\E_\infty$-ring over $R$.
We assume that $A$ is derived $I$-complete.
Let $a, b$ be integers with $a \leq b$.
\begin{enumerate}
    \item We put $S:=\pi_0(A)/I\pi_0(A)$.
    Let $K \in \Mod(A)$ be a module which is almost connective, i.e.\ there exists an integer $n$ such that the homotopy group $\pi_m(M)$ is trivial for every $m < n$.
    Then $K$ belongs to $\Perf_{[a, b]}(A)$ if and only if
    $K$ is derived $I$-complete and
    $K \otimes^\L_{A}S$ belongs to $\Perf_{[a, b]}(S)$.
    \item We put $A_m:=A \otimes^\L_R R/^\L (x^m_1, \dotsc,x^m_r)$ for an integer $m \geq 1$.
    We have
    \[
    \Perf(A) \overset{\sim}{\rightarrow} \lim_{m \geq 1} \Perf(A_m)
    \quad \text{in} \quad \Catinfty.
    \]
    Similarly, we have
    \[
    \Perf_{[a, b]}(A) \overset{\sim}{\rightarrow} \lim_{m \geq 1} \Perf_{[a, b]}(A_m) \quad \text{in} \quad \Catinfty.
    \]
\end{enumerate}
\end{prop}

\begin{proof}
(1) By \cite[Corollary 8.3.5.9]{SAG}, we see that $K$ belongs to $\Perf(A)$ if and only if 
$K$ is derived $I$-complete
and $K \otimes^\L_{A}S$ belongs to $\Perf(S)$.
We shall prove that a perfect module $K \in \Perf(A)$ belongs to $\Perf_{[a, b]}(A)$ if and only if $K \otimes^\L_{A}S$ belongs to $\Perf_{[a, b]}(S)$.
Let $K^\vee$ be a dual of $K$; see the proof of Lemma \ref{Lemma:Tor-amplitude} and the references therein.
We have $(K^\vee)^\vee \cong K$.
It follows that $K$ has Tor-amplitude in $[a, b]$ if and only if 
$K$ has Tor-amplitude $\leq b$ and $K^\vee$ has Tor-amplitude $\leq -a$ in the sense of \cite[Definition 7.2.4.21]{HA}.
Thus our claim follows from \cite[Corollary 8.3.5.8]{SAG}.

(2) The second equivalence follows from the first one by (1).
We prove the first assertion.
Let
\[
\Phi \colon \Mod(A) \to \lim_{m \geq 1} \Mod(A_m)
\]
denote the natural functor.
An object of $\lim_{m \geq 1} \Mod(A_m)$ can be identified with a family of objects $\{ K_m \}_{m \geq 1}$, where $ K_m \in \Mod(A_m)$,
together with equivalences
$K_{m+1}\otimes^\L_{A_{m+1}} A_m \overset{\sim}{\to} K_m$ ($m \geq 1$).
The functor $\Phi$ admits a right adjoint
\[
\Psi \colon \lim_{m \geq 1} \Mod(A_m) \to \Mod(A)
\]
which sends a family $\{ K_m \}_{m \geq 1}$ as above to $\lim_{m} K_m$, where we regard $K_m$ as an object of $\Mod(A)$.
We claim that if $K_m \in \Perf(A_m)$ for every $m \geq 1$, then $K:=\lim_{m} K_m$ belongs to $\Perf(A)$ and we have
\begin{equation}\label{equation:counit}
K \otimes^\L_A A_m \overset{\sim}{\to} K_m.
\end{equation}
In order to prove the claim, we may assume that $K_1$ is connective.
It then follows that $K_m$ is connective for every $m \geq 1$, and hence $K$ is also connective since $\pi_0(K_{m+1}) \to \pi_0(K_m)$ is surjective for every $m \geq 1$.
Now, \cite[Lemma 8.3.5.4]{SAG} implies that
the map $K \otimes^\L_A A_m \to K_m$ is an equivalence.
Since each $K_m$ is derived $I$-complete as an $A$-module, the limit $K$ is also derived $I$-complete.
Therefore $K$ is perfect by (1).

We now obtain the following adjunction
\[
\adj{F}{\Perf(A)}{\lim_{m \geq 1} \Perf(A_m)}{G}.
\]
It suffices to prove that the unit transformation $\id \to G \circ F$ and the counit transformation $F \circ G \to \id$ are equivalences.
For a $K \in \Perf(A)$, we see that $(G \circ F)(K)$ is isomorphic to the derived $I$-completion of $K$ as an object of $\Mod(R)$; see, for instance, \cite[Tag 0920]{SP}.
Since $K$ is derived $I$-complete by (1), we have $K \overset{\sim}{\to} (G \circ F)(K)$.
Finally, we shall prove $F \circ G \overset{\sim}{\to} \id$.
We note that a map in $\lim_{m \geq 1} \Perf(A_m)$ is an equivalence if and only if, for every $m$, its image in $\Perf(A_m)$ is an equivalence\footnote{This can be deduced from the fact that a natural transformation between two functors from a simplicial set to an $\infty$-category is an equivalence if and only if it is an objectwise equivalence (which can be found, e.g., in \cite[Corollary 3.5.12]{Cisinski}).
One can also check it by considering fiber sequences. We note that $\lim_{m \geq 1} \Perf(A_m)$ is also a limit in the $\infty$-category of stable $\infty$-categories (\cite[Theorem 1.1.4.4]{HA}).}.
Thus $(\ref{equation:counit})$ implies $F \circ G \overset{\sim}{\to} \id$.
\end{proof}

We conclude this section with the following result on derived quotients of perfectoid rings.

\begin{lem}\label{Lemma:derived quotient of perfectoid}
Let $R$ be a perfectoid ring and $\varpi \in R$ an element with $p \in (\varpi^p)$ such that $R$ is $\varpi$-complete.
Let $\varpi^\flat \in R^\flat$ be an element such that $(\varpi^\flat)^\sharp$ is a unit multiple of $\varpi$.
Let $A$ be a $\varpi$-complete perfectoid $R$-algebra.
Then we have an isomorphism of $\E_\infty$-rings
\[
A^\flat/^\L \varpi^\flat \cong A/^\L \varpi
\]
which is functorial in $A$.
\end{lem}
\begin{proof}
We may assume that $\varpi=(\varpi^\flat)^\sharp$.
Then $\varpi$ admits a $p$-th root,
and thus it suffices to prove that
$A^\flat/^\L (\varpi^\flat)^p \cong A/^\L (\varpi)^p$
functorially in $A$.
Since the map $\theta \colon W(R^\flat) \to R$ is surjective
and $p \in (\varpi^p)$,
there exists an element $x \in W(R^\flat)$ such that $p=-\theta(x)\varpi^p$.
Then $\xi:= p + x[\varpi^\flat]^p$ is a generator of $\Ker \theta$; see the proof of \cite[Lemma 3.10]{BMS}.
We consider two maps
\[
W(R^\flat)[S_1, S_2] \to W(A^\flat) \quad \text{defined by} \quad S_1 \mapsto p, \quad S_2 \mapsto [\varpi^\flat]^p
\]
and
\[
W(R^\flat)[T_1, T_2] \to W(A^\flat) \quad \text{defined by} \quad T_1 \mapsto \xi, \quad T_2 \mapsto [\varpi^\flat]^p.
\]
Since both $p$ and $\xi$ are non-zero divisors in $W(A^\flat)$, we have
\[
W(A^\flat) \otimes^\L_{W(R^\flat)[S_1, S_2]} W(R^\flat) \cong  A^\flat/^\L (\varpi^\flat)^p
\quad
\text{and}
\quad
W(A^\flat) \otimes^\L_{W(R^\flat)[T_1, T_2]} W(R^\flat) \cong A/^\L (\varpi)^p.
\]
Here $W(R^\flat)[S_1, S_2] \to W(R^\flat)$ is defined by $S_i \mapsto 0$ and similarly for $W(R^\flat)[T_1, T_2] \to W(R^\flat)$.
To conclude the proof, it suffices to observe that we have an isomorphism
$W(R^\flat)[T_1, T_2] \overset{\sim}{\to} W(R^\flat)[S_1, S_2]$ defined by $T_1 \mapsto S_1+xS_2$, $T_2 \mapsto S_2$, which is compatible with the above maps.
\end{proof}

\section{An analogue for derived quotients of perfect rings}\label{Section:An analogue for derived quotients of perfect rings}

\subsection{An analogue for derived quotients of perfect rings}\label{Subsection:An analogue for derived quotients of perfect rings}

In this section, we prove the following analogue of Theorem \ref{Theorem:descent for pefect rings} for derived quotients of perfect rings:

\begin{thm}\label{Theorem:descent for perfect complex for derived quotients}
Let $R$ be a perfect $\F_p$-algebra and $I=(x_1, \dotsc, x_r) \subset R$ a finitely generated ideal.
Let $a, b$ be integers with $a \leq b$.
Then the functors
\[
A \mapsto \mathcal{F}(A):=\Perf(A/^\L (x_1, \dotsc, x_r)) \quad \text{and} \quad  A \mapsto \mathcal{F}_{[a, b]}(A):=\Perf_{[a, b]}(A/^\L (x_1, \dotsc, x_r))
\]
from the category of perfect $R$-algebras to $\Catinfty$ satisfy $\arc$-hyperdescent (see also Definition \ref{Definition:hyperdescent}).
\end{thm}

We give two proofs of Theorem \ref{Theorem:descent for perfect complex for derived quotients}.
The first proof, given below, goes along the same line as that of Theorem \ref{Theorem:descent for pefect rings}; the notion of a \textit{descendable} map of $\E_\infty$-rings introduced by Mathew (\cite[Definition 3.18]{Mathew16}) plays a central role.
A key input is the fact that an \textit{$h$-cover} of perfect $\F_p$-algebras is descendable, which is proved by Bhatt--Scholze.

The second proof, which we learned from B.\ Bhatt, is given in the next subsection; we deduce Theorem \ref{Theorem:descent for perfect complex for derived quotients} from Theorem \ref{Theorem:descent for pefect rings}.
It may be easier than the first one for the reader who is familiar with some fundamental results on (almost) perfect modules provided in \cite{HA, SAG, SP}.

\begin{proof}
The functor $\mathcal{F}$ commutes with finite products (see, for instance, \cite[Lemma D.3.5.5]{SAG}) and filtered colimits (see \cite[Lemma 7.3.5.13]{HA}).
Then we see that the same results hold for $\mathcal{F}_{[a, b]}$.
To prove the theorem, it is enough to prove that $\mathcal{F}_{[a, b]}$ satisfies $\arc$-hyperdescent (for all $a, b$ with $a \leq b$).
Moreover, it suffices to prove that $\mathcal{F}_{[a, b]}$ satisfies $\arc$-descent since $\mathcal{F}_{[a, b]}(A)$ is equivalent to an $N$-category for every perfect $R$-algebra $A$ by Lemma \ref{Lemma:Tor-amplitude} (1), where $N:=b-a+r+1$.
We want to apply \cite[Proposition 4.8]{BM}.
For this, we need to prove the following assertions.

\begin{enumerate}
    \item[(i)]($v$-descent) For every $v$-cover $A \to B$ of perfect $R$-algebras (in the sense of \cite[Definition 2.1]{BS} or equivalently \cite[Definition 2.2]{Rydh2010}) with \v{C}ech conerve $A \to B^{\bullet}$, the functor
\[
\mathcal{F}_{[a, b]}(A) \to \lim_{\Delta} \mathcal{F}_{[a, b]}(B^\bullet)
\]
is an equivalence in $\Catinfty$ (or equivalently, in the full subcategory $\mathrm{Cat}_N \subset \Catinfty$ consisting of those $\infty$-categories which are equivalent to $N$-categories).
    \item[(ii)](aic-$v$-excision) For every valuation ring $V$ over $R$ with algebraically closed fraction field and every prime ideal $\p \in \Spec(V)$,
the square
\[
\xymatrix{
\mathcal{F}_{[a, b]}(V) \ar^-{}[r]  \ar[d]_-{} & \mathcal{F}_{[a, b]}(V/\p)  \ar[d]_-{}  \\
\mathcal{F}_{[a, b]}(V_\p)  \ar[r]^-{} & \mathcal{F}_{[a, b]}(\kappa(\p)) 
}
\]
is a pullback square in $\Catinfty$ (or equivalently, in $\mathrm{Cat}_N$).
\end{enumerate}

We first prove (i).
We can write $B$ as a colimit of a filtered system $\{ B_{0, i} \}_{i \in I}$ of finitely presented $A$-algebras.
Let $B_i:= (B_{0, i})_{\mathrm{perf}} := \colim_{x \mapsto x^p} B_{0, i}$ be the perfection of $B_{0, i}$.
It follows that $A \to B_i$ is an $h$-cover (in the sense of \cite[Definition 11.1]{BS}) and hence it is descendable in the sense of \cite[Definition 3.18]{Mathew16} (or equivalently \cite[Definition 11.14]{BS}) by \cite[Theorem 11.27]{BS}.
The base change 
\[
f_i \colon A/^\L (x_1, \dotsc, x_r) \to B_i/^\L (x_1, \dotsc, x_r)
\]
is also descendable.
Let $A \to B^{\bullet}_i$ denote the \v{C}ech conerve of $A \to B_i$ taken in the category $\Ring$ of commutative rings.
The (derived) \v{C}ech conerve of $f_i$ taken in the $\infty$-category $\CAlg$ of $\E_\infty$-algebras
is isomorphic to the base change
\[
A/^\L (x_1, \dotsc, x_r) \to B^\bullet_i/^\L (x_1, \dotsc, x_r)
\]
of
$A \to B^{\bullet}_i$ along $A \to A/^\L (x_1, \dotsc, x_r)$ by \cite[Lemma 3.16 or Proposition 11.6]{BS}.
Therefore, by \cite[Proposition 3.22]{Mathew16}, we have an equivalence of symmetric monoidal $\infty$-categories for every $i \in I$
\[
\Mod(A/^\L (x_1, \dotsc, x_r)) \overset{\sim}{\to} \lim_{\Delta} \Mod(B^\bullet_i/^\L (x_1, \dotsc, x_r)), \quad \text{and hence} \quad 
\mathcal{F}(A) \overset{\sim}{\to} \lim_{\Delta} \mathcal{F}(B^\bullet_i);
\]
here, for the second equivalence, we use \cite[Proposition 4.6.1.11]{HA} and the fact that an object of $\Mod(A/^\L (x_1, \dotsc, x_r))$ is perfect if and only if it is dualizable (see \cite[Proposition 2.7.28]{DAGVIII}).
By Lemma \ref{Lemma:Tor-amplitude} (2), this implies
\[
\mathcal{F}_{[a, b]}(A) \overset{\sim}{\to} \lim_{\Delta} \mathcal{F}_{[a, b]}(B^\bullet_i)
\]
for every $i \in I$.
Since we have $\colim_{I}B_i \cong B$ and $\mathcal{F}_{[a, b]}(A)$ is equivalent to an $N$-category for every $A$,
it follows from \cite[Lemma 3.7]{BM} that the functor
$\mathcal{F}_{[a, b]}(A) \to \lim_{\Delta} \mathcal{F}_{[a, b]}(B^\bullet)$ is an equivalence.

Next, we prove (ii).
As remarked in the proof of \cite[Theorem 5.16]{BM},
the map $V \to W:=V_\p \times V/\p$ is descendable.
The base change
$
V/^\L (x_1, \dotsc, x_r) \to W/^\L (x_1, \dotsc, x_r)
$
is also descendable.
Thus, similarly to the proof of (i),
we have
\[
\mathcal{F}_{[a, b]}(V) \overset{\sim}{\to} \lim_{\Delta} \mathcal{F}_{[a, b]}(W^\bullet)
\]
for the \v{C}ech conerve
$V \to W^{\bullet}$ of $V \to W$ taken in $\Ring$.
As in the proof of \cite[Theorem 5.16]{BM}, one can prove that the right hand side of the above equivalence is a limit of the diagram
\[
\xymatrix{
& \mathcal{F}_{[a, b]}(V/\p)  \ar[d]_-{}  \\
\mathcal{F}_{[a, b]}(V_\p)  \ar[r]^-{} & \mathcal{F}_{[a, b]}(\kappa(\p)).
}
\]
This proves (ii).

With assertions (i) and (ii), we can now apply
the argument as in the proof of \cite[Proposition 4.8]{BM} to $\mathcal{F}_{[a, b]}$ (or rather the functor $A \mapsto \mathcal{F}_{[a, b]}(A_{\mathrm{perf}})$ from the category of $R$-algebras to $\mathrm{Cat}_N$)
to conclude that
$\mathcal{F}_{[a, b]}$ satisfies $\arc$-descent.
The proof of Theorem \ref{Theorem:descent for perfect complex for derived quotients} is complete.
\end{proof}

The following corollary is a key ingredient in the proof of Theorem \ref{Theorem:descent for perfectoid}.

\begin{cor}\label{Corollary:descent for derived quotient of perfectoid}
Let $R$ be a perfectoid ring and $\varpi \in R$ an element with $p \in (\varpi^p)$ such that $R$ is $\varpi$-complete.
Let $m \geq 1$ be an integer.
The functors
\[
A \mapsto \Perf(A/^\L \varpi) \quad \text{and} \quad A \mapsto \Perf(A^\flat/^\L (\varpi^\flat)^m)
\]
from the category $\mathcal{C}_{R, \varpi}$ of $\varpi$-complete perfectoid $R$-algebras to $\Catinfty$ satisfy $\varpi$-complete $\arc$-hyperdescent.
The same statement holds for perfect modules with Tor-amplitude in $[a, b]$ for all $a, b$ with $a \leq b$.
\end{cor}

\begin{proof}
By Proposition \ref{Proposition:tilting and arc cover} and Lemma \ref{Lemma:derived quotient of perfectoid},
it suffices to prove that
the functors
\[
A \mapsto \Perf(A/^\L (\varpi^\flat)^m) \quad \text{and} \quad A \mapsto \Perf_{[a, b]}(A/^\L (\varpi^\flat)^m)
\]
from the category $\mathcal{C}_{R^\flat, \varpi^\flat}$ of $\varpi^\flat$-complete perfect $R^\flat$-algebras to $\Catinfty$ satisfy $\varpi^\flat$-complete $\arc$-hyperdescent.
Moreover, it is enough to prove that
the second functor satisfies
$\varpi$-complete $\arc$-descent; compare the first paragraph of the (first) proof of Theorem \ref{Theorem:descent for perfect complex for derived quotients}.
This functor preserves finite products.
Let $A \to B$ be a $\varpi^\flat$-complete $\arc$-cover in $\mathcal{C}_{R^\flat, \varpi^\flat}$.
Since $A \to C:=B \times A[1/{\varpi^\flat}]$ is an $\arc$-cover,
we obtain
\[
\Perf_{[a, b]}(A/^\L (\varpi^\flat)^m) \overset{\sim}{\to} \lim_{\Delta} \Perf_{[a, b]}(C^{\bullet}/^\L (\varpi^\flat)^m)
\]
by Theorem \ref{Theorem:descent for perfect complex for derived quotients}, where $A \to C^\bullet$ is the \v{C}ech conerve of $A \to C$ taken in the category of perfect $R^\flat$-algebras.
Let $A \to B^{\bullet}$ be the (completed) \v{C}ech conerve
of $A \to B$ taken in $\mathcal{C}_{R^\flat, \varpi^\flat}$.
We have
$C^{\bullet}/^\L (\varpi^\flat)^m \cong B^{\bullet}/^\L (\varpi^\flat)^m$, and thus
\[
\Perf_{[a, b]}(C^{\bullet}/^\L (\varpi^\flat)^m) \cong \Perf_{[a, b]}(B^{\bullet}/^\L (\varpi^\flat)^m).
\]
In conclusion, we have
$
\Perf_{[a, b]}(A/^\L (\varpi^\flat)^m) \overset{\sim}{\to} \lim_{\Delta} \Perf_{[a, b]}(B^{\bullet}/^\L (\varpi^\flat)^m),
$
which completes the proof.
\end{proof}

\subsection{The second proof of Theorem \ref{Theorem:descent for perfect complex for derived quotients}}\label{Subsection:The second proof}

In this subsection, we prove the following proposition, by which one can easily see that Theorem \ref{Theorem:descent for pefect rings} implies 
Theorem \ref{Theorem:descent for perfect complex for derived quotients}.

\begin{prop}\label{Proposition:reduction argument}
Let $A \to B^\bullet$ be an augmented cosimplicial object in $\Ring$ and $I=(x_1, \dotsc, x_r) \subset A$ a finitely generated ideal such that
$\Spec B^0/IB^0 \to \Spec A/I$
is surjective.
Let
$A/^\L(x_1, \dotsc, x_r) \to B^\bullet/^\L(x_1, \dotsc, x_r)$
be
the augmented cosimplicial object in $\CAlg$
obtained by base change.
If
$
\Perf(A) \overset{\sim}{\to} \lim_{\Delta} \Perf(B^{\bullet}),
$
then we have
\[
F \colon \Perf(A/^\L(x_1, \dotsc, x_r)) \overset{\sim}{\to} \lim_{\Delta} \Perf(B^{\bullet}/^\L(x_1, \dotsc, x_r)).
\]
\end{prop}

\begin{proof}
To simplify the notation, we write
\[
\overline{A}:=A/^\L(x_1, \dotsc, x_r) \quad \text{and} \quad \overline{B}^n:=B^n/^\L(x_1, \dotsc, x_r).
\]
Let
\[
\Phi \colon \Mod(\overline{A}) \to \lim_{\Delta} \Mod(\overline{B}^{\bullet})
\]
denote the natural functor.
This functor $\Phi$ admits a right adjoint
\[
\Psi \colon \lim_{\Delta} \Mod(\overline{B}^{\bullet}) \to \Mod(\overline{A}).
\]
An object $\{ K^\bullet \}$ of $\lim_{\Delta} \Mod(\overline{B}^{\bullet})$ gives rise to a natural functor $\Delta \to \Mod(\overline{A})$
such that the image $K^n$ of $[n] \in \Delta$ is isomorphic to that of $\{ K^\bullet \}$ in $\Mod(\overline{B}^{n})$ regarded as an $\overline{A}$-module.
The functor $\Psi$ sends $\{ K^\bullet \}$ to $\lim_{[n] \in \Delta} K^n$.

We first prove the fully faithfulness of $F$.
Let $\mathcal{D} \subset \Mod(\overline{A})$ be the full subcategory spanned by those objects $K$ such that
the unit map $K \to (\Psi \circ \Phi)(K)$ is an isomorphism.
It is clear that $\mathcal{D}$ is a stable subcategory which is closed under retracts.
Using the fully faithfulness of
$
\Perf(A) \to \lim_{\Delta} \Perf(B^{\bullet})
$
and using the fiber sequences
\[
A/^\L(x_1, \dotsc, x_i) \overset{\times x_{i+1}}{\to} A/^\L(x_1, \dotsc, x_i) \to A/^\L(x_1, \dotsc, x_{i+1})
\]
for $0 \leq i \leq r-1$ inductively, where $A/^\L(x_1, \dotsc, x_i):=A$ if $i=0$, we see that $\overline{A} \in \mathcal{D}$.
Thus, by the definition of $\Perf(\overline{A})$, we have
$
\Perf(\overline{A}) \subset \mathcal{D}.
$
In particular, it follows that $F$ is fully faithful.

We next prove that $F$ is essentially surjective.
Let $\{ K^\bullet \}$ be an object of $\lim_{\Delta} \Perf(\overline{B}^{\bullet})$
and
let $K^n \in \Perf(\overline{B}^n)$ denote its image.
It suffices to prove that
$L:=\Psi(\{ K^\bullet \})$ belongs to $\Perf(\overline{A})$ and the canonical map
$L \otimes^\L_{\overline{A}} \overline{B}^n \to K^n$
is an isomorphism.
(See also the proof of Proposition \ref{Proposition:derived complete and reduction modulo I} (2).)
Since $\overline{B}^n$ is perfect as a $B^n$-module,
it follows that $K^n$ is perfect over $B^n$ as well.
Thus, by our assumption (and a similar argument as in the previous paragraph), we see that $L$ is perfect over $A$ and we have
$
L \otimes^\L_{\overline{A}} \overline{B}^n \cong L \otimes^\L_{A} B^n \cong K^n.
$
It remains to show that $L$ is perfect over $\overline{A}$.
For this, we remark that,
since $\overline{A}$ is perfect as an $A$-module,
the $\overline{A}$-module $L$ is almost perfect in the sense of \cite[Definition 7.2.4.10]{HA} by Lemma \ref{Lemma:almost perfect push} below.
Therefore, by \cite[Tag 068W]{SP} and \cite[Proposition 2.7.3.2 (d)]{SAG} (see also \cite[Corollary 8.6.4.3]{SAG}),
it suffices to show that
$L \otimes^\L_{\overline{A}} \pi_0(\overline{A})/\mathfrak{m}$
is perfect over $\pi_0(\overline{A})/\mathfrak{m}$ for every maximal ideal $\mathfrak{m} \subset \pi_0(\overline{A})$.
This follows from the fact that
$L \otimes^\L_{\overline{A}} \overline{B}^0 \cong K^0$
is perfect over $\overline{B}^0$ and
the assumption that
$\Spec B^0/IB^0 \to \Spec A/I$
is surjective.
\end{proof}

The following easy lemma is used in the proof of Proposition \ref{Proposition:reduction argument}.

\begin{lem}\label{Lemma:almost perfect push}
Let $A \to B$ be a map of connective $\E_\infty$-rings.
Assume that $B$ is almost perfect as an $A$-module (in the sense of \cite[Definition 7.2.4.10]{HA}).
Then a $B$-module $M$ is almost perfect if and only if $M$ is almost perfect as an $A$-module.
\end{lem}

\begin{proof}
Let $n \geq 0$ be an integer.
We prove that a connective $B$-module $M$ is perfect to order $n$ in the sense of \cite[Definition 2.7.0.1]{SAG} if and only if $M$ is perfect to order $n$ as an $A$-module, from which the lemma follows (see \cite[Remark 2.7.0.2]{SAG}).
The ``only if'' direction is proved in \cite[Proposition 2.7.3.3]{SAG}.
The ``if'' direction can be proved in a similar way; we provide the proof for the reader's convenience.
(See also \cite[Tag 064Z]{SP} for the case where $A$ and $B$ are discrete.)

We proceed by induction on $n$.
Since a connective $B$-module $M$ is perfect to order $0$ if and only if $\pi_0(M)$ is finitely generated over $\pi_0(B)$ by \cite[Proposition 2.7.2.1]{SAG}, the case where $n=0$ is clear.
Assume $n>0$.
Since $\pi_0(M)$ is finitely generated over $\pi_0(B)$, there exists a map
$f \colon B^{\oplus m} \to M$ of $B$-modules such that the induced map $\pi_0(B^{\oplus m}) \to \pi_0(M)$ is surjective.
Let $N$ be a fiber of $f$, which is a connective $B$-module.
Since both $B$ and $M$ are perfect to $n$ as $A$-modules, we see that $N$ is perfect to $n-1$ over $A$
by \cite[Remark 2.7.0.7]{SAG}, and hence it is perfect to $n-1$ over $B$ by the induction hypothesis.
By \cite[Remark 2.7.0.7]{SAG} again, it follows that $M$ is perfect to $n$ over $B$.
\end{proof}

\section{Proofs of main results}\label{Section:Proofs of main results}

\subsection{Proof of Theorem \ref{Theorem:descent for perfectoid}}\label{Subsection:Proof of Theorem:descent for perfectoid}

We prove Theorem \ref{Theorem:descent for perfectoid}.
By Proposition \ref{Proposition:derived complete and reduction modulo I} (2) and Corollary \ref{Corollary:descent for derived quotient of perfectoid},
the functor $A \mapsto \Perf(A^\flat)$ satisfies $\varpi$-complete $\arc$-hyperdescent.

We prove the assertions for the functors
$A \mapsto \Perf(W(A^\flat))$
and
$A \mapsto \Perf(A)$
simultaneously.
For this,
let 
$\mathcal{G}$
denote
the functor $A \mapsto W(A^\flat)$ or $A \mapsto A$
from $\mathcal{C}_{R, \varpi}$ to $\Ring$,
and let $\alpha \in \mathcal{G}(A)$ be $p$ or $\varpi$, respectively.
Then, for every integer $m \geq 1$, let
$\mathcal{G}_m \colon \mathcal{C}_{R, \varpi} \to \CAlg$ be the functor defined by $A \mapsto \mathcal{G}(A)/^\L \alpha^m$.
By Proposition \ref{Proposition:derived complete and reduction modulo I} (2),
it suffices to prove that the functor
$A \mapsto \Perf(\mathcal{G}_m(A))$
satisfies $\varpi$-complete $\arc$-hyperdescent for every integer $m \geq 1$.
This functor preserves finite products.
We want to show that the functor
\[
F \colon \Perf(\mathcal{G}_m(A)) \to \lim_{\Delta} \Perf(\mathcal{G}_m(B^\bullet))
\]
is an equivalence
for any $\varpi$-complete $\arc$-hypercover $A \to B^\bullet$.
We first prove the fully faithfulness of $F$.
The natural functor
\[
\Phi_m \colon \Mod(\mathcal{G}_m(A)) \to \lim_{\Delta} \Mod(\mathcal{G}_m(B^\bullet))
\]
admits a right adjoint
\[
\Psi_m \colon \lim_{\Delta} \Mod(\mathcal{G}_m(B^\bullet)) \to \Mod(\mathcal{G}_m(A)).
\]

\begin{lem}\label{Lemma:fully fatihful}
Let $\mathcal{D} \subset \Mod(\mathcal{G}_m(A))$ be the full subcategory spanned by those objects $K$ such that
the unit map $K \to (\Psi_m \circ \Phi_m)(K)$ is an isomorphism.
Then we have $\Perf(\mathcal{G}_m(A)) \subset \mathcal{D}$.
In particular, the functor $F$ is fully faithful.
\end{lem}

\begin{proof}
As in the proof of Proposition \ref{Proposition:reduction argument},
it suffices to show
$\mathcal{G}_m(A) \in \mathcal{D}$.
We proceed by induction on $m$.
Assume $m=1$.
In the beginning of this section and Corollary \ref{Corollary:descent for derived quotient of perfectoid},
we have proved that the functor $F$ is an equivalence (and in particular fully faithful), which implies the result.
When $m > 1$, the result follows from the fiber sequence
\begin{equation}\label{equation:inductive fiber sequence}
\mathcal{G}_{m-1}(A) \overset{\times \alpha}{\to} \mathcal{G}_m(A) \to \mathcal{G}_1(A)
\end{equation}
and the induction hypothesis.
\end{proof}

We shall prove that $F$ is essentially surjective.
Again, we proceed by induction on $m$.
The case $m=1$ has already been established in the beginning of the proof and Corollary \ref{Corollary:descent for derived quotient of perfectoid}.
Assume $m >1$.
Let $\{ K^\bullet_m \}$ be an object of $\lim_{\Delta} \Perf(\mathcal{G}_m(B^\bullet))$
and 
let $K^n_m \in \Perf(\mathcal{G}_m(B^n))$ denote its image.
It suffices to prove that
$L_m:=\Psi_m(\{ K^\bullet_m \})$ belongs to $\Perf(\mathcal{G}_m(A))$ and the canonical map
\begin{equation}\label{equation:desired isomorphism}
    L_m \otimes^\L_{\mathcal{G}_m(A)} \mathcal{G}_m(B^n) \to K^n_m
\end{equation}
is an isomorphism.
(See also the proofs of Proposition \ref{Proposition:derived complete and reduction modulo I} (2) and Proposition \ref{Proposition:reduction argument}.)

By the induction hypothesis,
the image $\{ K^\bullet_l \}$ of $\{ K^\bullet_m \}$ in $\lim_{\Delta} \Perf(\mathcal{G}_l(B^\bullet))$
corresponds to a perfect module
$L_l \in \Perf(\mathcal{G}_l(A))$
for every $l \leq m-1$.
By Lemma \ref{Lemma:fully fatihful}, we have
$L_l \cong \Psi_l(\{ K^\bullet_l \})$.
We claim that $L_m$ is almost connective and
\begin{equation}\label{equation:isom special case}
L_m\otimes^\L_{\mathcal{G}_m(A)} \mathcal{G}_1(A) \overset{\sim}{\to} L_1.
\end{equation}
Our claim
together with Proposition \ref{Proposition:derived complete and reduction modulo I} (1)
shows that $L_m \in \Perf(\mathcal{G}_m(A))$.
It also implies that
the map $(\ref{equation:desired isomorphism})$ is an isomorphism.
Indeed,
the base change of $(\ref{equation:desired isomorphism})$ along $\mathcal{G}_m(B^n) \to \mathcal{G}_1(B^n)$
is an isomorphism since it can be identified with the base change of $(\ref{equation:isom special case})$ along
$\mathcal{G}_1(A) \to \mathcal{G}_1(B^n)$.
Then, using $(\ref{equation:inductive fiber sequence})$ repeatedly, we conclude that the map $(\ref{equation:desired isomorphism})$ is also an isomorphism.

We shall prove our claim.
Using $(\ref{equation:inductive fiber sequence})$ and the functor $\Psi_m$,
we obtain a fiber sequence
$
L_{m-1} \to L_m \to L_1.
$
Since $L_{m-1}$ and $L_1$ are almost connective, it follows that $L_m$ is also almost connective.
It remains to show $(\ref{equation:isom special case})$.
We have a natural fiber sequence
$\mathcal{G}_{1}(A) \to \mathcal{G}_m(A) \to \mathcal{G}_{m-1}(A)$.
Moreover, from the following fiber sequence
\begin{equation}\label{equation:4-4}
    \mathcal{G}(A) \overset{\times \alpha}{\to} \mathcal{G}(A) \to \mathcal{G}_1(A),
\end{equation}
we obtain a fiber sequence
$\mathcal{G}_{m}(A) \to \mathcal{G}_m(A) \to \mathcal{G}_m(A) \otimes^\L_{\mathcal{G}(A)} \mathcal{G}_1(A)$.
Then, $(\ref{equation:inductive fiber sequence})$ and these fiber sequences induce a fiber sequence
\begin{equation}\label{equation:core fiber sequence}
\mathcal{G}_1(A)[1] \to \mathcal{G}_m(A)\otimes^\L_{\mathcal{G}(A)} \mathcal{G}_1(A) \to \mathcal{G}_1(A),
\end{equation}
where $[1]$ denotes the shift by $1$.
From $(\ref{equation:core fiber sequence})$,
we have a natural map of fiber sequences
\[
\xymatrix{
L_m\otimes^\L_{\mathcal{G}_m(A)} \mathcal{G}_1(A)[1] \ar^-{}[r]  \ar[d]_-{} & L_m \otimes^\L_{\mathcal{G}(A)} \mathcal{G}_1(A) \ar[d]_-{\id} \ar^-{}[r] & L_m\otimes^\L_{\mathcal{G}_m(A)} \mathcal{G}_1(A) \ar[d]_-{} \\
L_1[1] \ar[r]^-{} & L_m \otimes^\L_{\mathcal{G}(A)} \mathcal{G}_1(A) \ar^-{}[r] & L_1.
}
\]
The first line is obtained by applying $L_m \otimes^\L_{\mathcal{G}_m(A)} -$ to $(\ref{equation:core fiber sequence})$.
The second line is obtained by applying $K^n_m \otimes^\L_{\mathcal{G}_m(A)} -$ to $(\ref{equation:core fiber sequence})$ and then using the functor $\Psi_m$.
Here we use the fact that the functor
$- \otimes^\L_{\mathcal{G}(A)} \mathcal{G}_1(A) \colon \Mod(\mathcal{G}(A)) \to \Mod(\mathcal{G}_1(A))$ preserves limits (which can be checked using the fiber sequence $(\ref{equation:4-4})$).
We shall prove that the map
$L_m\otimes^\L_{\mathcal{G}_m(A)} \mathcal{G}_1(A) \to L_1$
induces an isomorphism on the $k$-th homotopy groups for every $k$.
Since $L_m$ is almost connective, so is $L_m\otimes^\L_{\mathcal{G}_m(A)} \mathcal{G}_1(A)$.
In particular, the assertion holds for a sufficiently small $k$.
Then, by induction on $k$ and using the above diagram,
we see that the assertion holds for all $k$.
This concludes the proof of $(\ref{equation:isom special case})$ and hence the proof of the essential surjectivity of $F$.

The proof of Theorem \ref{Theorem:descent for perfectoid} is now complete.

\subsection{$\varpi$-complete $\arc$-descent for finite projective modules}\label{Subsection:Descent for finite projective modules}

For future reference, we record a $\varpi$-complete $\arc$-descent result for finite projective modules over perfectoid rings, which is a consequence of Theorem \ref{Theorem:descent for perfectoid}.

We first fix some notation.
Let $A \to B^\bullet$ be an augmented cosimplicial object in $\Ring$.
For any $0 \leq i \leq 1$,
we denote by $p_i \colon B^0 \to B^1$ the map corresponding to the injection $[0]\cong \{i \} \hookrightarrow [1]$ in $\Delta$.
Similarly, for any $0 \leq i < j \leq 2$,
we denote by $p_{i, j} \colon B^1 \to B^2$ the map corresponding to the injection $[1] \cong \{i, j \}  \hookrightarrow [2]$ in $\Delta$.
Let $\Vect(A)$ be the category of finite projective modules over $A$ and $\DD(B^\bullet)$ the category of pairs $(M, \sigma)$ where $M \in \Vect(B^0)$ is a finite projective module over $B^0$ and $\sigma \colon p^*_0 M \overset{\sim}{\to} p^*_1 M$ is
an isomorphism in $\Vect(B^1)$ satisfying the
usual cocycle condition
$p^*_{0, 2}\sigma=p^*_{1, 2}\sigma \circ p^*_{0, 1}\sigma$.
Here for a map of commutative rings $f \colon R \to S$ and a module $M$ over $R$, we denote by $f^*M:=M\otimes_A B$ the base change of $M$ along $f$.
We have a natural functor
\[
\Vect(A) \to \DD(B^\bullet).
\]

\begin{cor}\label{Corollary:finite projective modules}
Let $R$ be a perfectoid ring and $\varpi \in R$ an element with $p \in (\varpi^p)$ such that $R$ is $\varpi$-complete.
Let $A \to B^\bullet$ be a $\varpi$-complete $\arc$-hypercover in $\mathcal{C}_{R, \varpi}$.
Then we have the equivalences of categories
\[
\Vect(A^\flat) \overset{\sim}{\to} \DD((B^\bullet)^\flat), \quad \Vect(W(A^\flat)) \overset{\sim}{\to} \DD(W((B^\bullet)^\flat)), \quad \text{and} \quad \Vect(A) \overset{\sim}{\to} \DD(B^\bullet).
\]
\end{cor}

\begin{proof}
We only prove $\Vect(A) \overset{\sim}{\to} \DD(B^\bullet)$; the other statements can be proved similarly.
First, note that $\Vect(A) \cong \Perf_{[0, 0]}(A)$.
By Theorem \ref{Theorem:descent for perfectoid}, Lemma \ref{Lemma:Tor-amplitude} (2), and Proposition \ref{Proposition:derived complete and reduction modulo I} (1), we have
$
\Vect(A) \overset{\sim}{\to} \lim_{\Delta} \Vect(B^\bullet).
$
Let $\Delta_{s, \leq n}$ be the subcategory of $\Delta$ whose objects are $[m]=\{0, 1, \dotsc, m \}$ for $0 \leq m \leq n$ and morphisms are given by injective order preserving maps.
Since $\Vect(B^n)$ is a $1$-category for any $n$, it follows that
$\lim_{\Delta} \Vect(B^\bullet) \cong \lim_{\Delta_{s, \leq 2}} \Vect(B^\bullet)$; see, for instance, \cite[Proposition 4.3.5]{DAGVIII}.
(Here we use the dual of \cite[Proposition 4.3.5]{DAGVIII}. Although it is claimed there that the geometric realization of $X_\bullet$ is isomorphic to a colimit of the diagram $X_\bullet\vert_{\Delta^{op}_{s, \leq n+1}}$,
its proof shows that it is also isomorphic to a colimit of the diagram $X_\bullet\vert_{\Delta^{op}_{s, \leq n}}$; compare \cite[Lemma 5.5.6.17]{HTT}.)
One can check that
$\DD(B^\bullet) \cong \lim_{\Delta_{s, \leq 2}} \Vect(B^\bullet)$, which concludes the proof.
\end{proof}

\begin{rem}\label{Remark:Henkel's thesis}
By the proof of Theorem \ref{Theorem:descent for perfectoid} and the argument in the proof of Corollary \ref{Corollary:finite projective modules},
we also have an equivalence of categories
$\Vect(W_n(A^\flat)) \overset{\sim}{\to} \DD(W_n((B^\bullet)^\flat))$
for every integer $n \geq 1$, where $W_n(A^\flat):=W(A^\flat)/p^n$.
This implies that \cite[Conjecture A]{Henkel} holds.

In \cite[Section 4.3.3]{Henkel},
Henkel proved some $p$-complete $\arc$-descent results for
$\mathrm{BK}_n$-modules for perfectoid rings assuming this conjecture.
Here $\mathrm{BK}_n$-modules are ``truncated'' analogues of (minuscule) Breuil--Kisin--Fargues modules; see \cite[Chapter 2]{Henkel} for details.
For example, he proved that \cite[Conjecture A]{Henkel} implies that
the functor sending a perfectoid $R$-algebra $A$ to the groupoid $\mathrm{BK}_n(A)$ of $\mathrm{BK}_n$-modules for $A$ is a stack with respect to the $p$-complete $\arc$-topology; see \cite[Theorem 4.3.15]{Henkel}.
\end{rem}

\section{The classification of $p$-divisible groups over perfectoid rings}\label{Section:The classification of p-divisible groups over perfectoid rings}

In this section, we discuss the classification of $p$-divisible groups over perfectoid rings (Theorem \ref{Theorem:Classifition p-divisible group Intro}).
We follow the approach of Scholze--Weinstein \cite[Theorem 17.5.2]{Scholze-Weinstein}.

Let $A$ be a perfectoid ring.
Let $\varphi$ be the Frobenius automorphism of $W(A^\flat)$.
For a generator $\xi \in \Ker \theta$, we put $\widetilde{\xi}:=\varphi(\xi)$.
Recall that a \textit{minuscule Breuil--Kisin--Fargues module} for $A$ is a finite projective module $M$ over $W(A^\flat)$ with a $W(A^\flat)$-linear map
$
F_M \colon \varphi^*M \to M
$
such that the cokernel of $F_M$ is killed by $\widetilde{\xi}$.
Note that, since $\widetilde{\xi}$ is a non-zero divisor, the condition on $F_M$ is equivalent to the existence of a $W(A^\flat)$-linear map
$
V_M \colon M \to \varphi^*M
$
such that $F_M \circ V_M= \widetilde{\xi}$.
Moreover, for a fixed $\xi$,
such a map $V_M$ is uniquely determined.

We begin by recalling the following special case:

\begin{thm}\label{Theorem:classification over valuation ring}
Let $A$ be a perfectoid ring.
If $A$ satisfies one of the following conditions, then
there exists an anti-equivalence $\mathcal{M}$ from the category of $p$-divisible groups over $A$ to the category of minuscule Breuil--Kisin--Fargues modules for $A$.
\begin{enumerate}
    \item (Berthelot, Gabber, Lau) $A$ is a perfect ring over $\F_p$. 
    \item (Fargues, Scholze--Weinstein) $A$ is the ring of integers $\O_C$ of an algebraically closed non-archimedean extension $C$ of $\Q_p$.
\end{enumerate}
\end{thm}

\begin{proof}
(1) For a $p$-divisible group $\mathcal{G}$ over $A$,
let $\mathcal{M}(\mathcal{G}):=\mathbb{D}(\mathcal{G})(W(A))$ be the evaluation on the divided power extension $W(A) \to A$ of the contravariant Dieudonn\'e crystal $\mathbb{D}(\mathcal{G})$ defined in \cite[D\'efinition 3.3.6]{BBM}.
This construction induces an anti-equivalence from the category of $p$-divisible groups over $A$ to the category of minuscule Breuil--Kisin--Fargues modules for $A$;
this fact is proved by Berthelot \cite[Corollaire 3.4.3]{Berthelot} (see also \cite[Proposition 4.3.4]{Berthelot}) for a perfect valuation ring over $\F_p$, and it is proved by Gabber and Lau \cite[Theorem 6.4]{Lau} independently for a general perfect ring over $\F_p$.

(2) See \cite[Theorem 14.4.1]{Scholze-Weinstein}, which is based on \cite[Theorem B]{Scholze-Weinstein2013}.
In this paper, for a $p$-divisible group $\mathcal{G}$ over $\O_C$, we define $\mathcal{M}(\mathcal{G})$ to be the $W(\O^\flat_C)$-linear dual of the Breuil--Kisin--Fargues module attached to $\mathcal{G}$ given in \cite[Theorem 14.4.1]{Scholze-Weinstein}.
\end{proof}

We now deduce the general case from Theorem \ref{Theorem:classification over valuation ring} by using Corollary \ref{Corollary:finite projective modules}:

\begin{thm}[{Lau, Scholze--Weinstein}]\label{Theorem:classification over perfectoid ring}
For each perfectoid ring $A$, there exists an anti-equivalence of categories
\[
\mathcal{M}_A \colon \{ p\text{-divisible groups over} \ A  \} \overset{\sim}{\to} \{ \text{minuscule Breuil--Kisin--Fargues modules for} \ A \}
\]
satisfying the following properties:
\begin{itemize}
    \item $\mathcal{M}_A$ is compatible with base change in $A$.
    \item If $p=0$ in $A$ or $A=\O_C$ for an algebraically closed non-archimedean extension $C/\Q_p$, then $\mathcal{M}_A$ coincides with the anti-equivalence given in Theorem \ref{Theorem:classification over valuation ring}.
\end{itemize}
\end{thm}

\begin{proof}
As in the proof of \cite[Theorem 17.5.2]{Scholze-Weinstein},
Theorem \ref{Theorem:classification over valuation ring} implies that, for each perfectoid ring $A=\prod_{i}V_i$ which is a product of perfectoid valuation rings $V_i$ of rank $\leq 1$ with algebraically closed fraction fields,
we have an anti-equivalence $\mathcal{M}_A$ satisfying the above conditions.

Let $A$ be a general perfectoid ring.
By \cite[Lemma 2.2.3]{CS} (and Example \ref{Example:pefect rings case}), there exists a $p$-complete $\arc$-hypercover $A \to B^\bullet$ whose terms $B^n$ are products of perfectoid valuation rings of rank $\leq 1$ with algebraically closed fraction fields.
By Corollary \ref{Corollary:finite projective modules} and \cite[Proposition 1.1]{deJong},
the category of $p$-divisible groups over $A$ is equivalent to the category of $p$-divisible groups over $B^0$ with descent data (defined in the same way as the category $\DD(B^\bullet)$ in Section \ref{Subsection:Descent for finite projective modules}).
Applying Corollary \ref{Corollary:finite projective modules} again, we see that the same statement holds for minuscule Breuil--Kisin--Fargues modules.
Thus we can obtain an anti-equivalence $\mathcal{M}_A$
from the category of $p$-divisible groups over $A$ to the category of minuscule Breuil--Kisin--Fargues modules for $A$
by using these equivalences and anti-equivalences $\mathcal{M}_{B^n}$.
One can check that $\mathcal{M}_A$ does not depend (up to canonical equivalence) on the choice of $A \to B^\bullet$, and $\mathcal{M}_A$ is compatible with base change in $A$.
We also note that, if $p=0$ in $A$, then $\mathcal{M}_A$ coincides with the anti-equivalence given in Theorem \ref{Theorem:classification over valuation ring}
since the formation of $\mathbb{D}(\mathcal{G})(W(A))$ is compatible with base change in $A$.

The proof of Theorem \ref{Theorem:classification over perfectoid ring} is complete.
\end{proof}

\begin{rem}\label{Remark: prismatic Dieudonne}
The conditions in Theorem \ref{Theorem:classification over perfectoid ring} determine the anti-equivalences
$\mathcal{M}_A$ uniquely (up to canonical equivalences).
Ansch\"utz--Le Bras give a cohomological description of $\mathcal{M}_A$
using the prismatic site of a perfectoid ring $A$ developed by Bhatt--Scholze; see \cite{AL} for details.
\end{rem}

\subsection*{Acknowledgements}
The author would like to thank K\k{e}stutis \v{C}esnavi\v{c}ius, Tetsushi Ito, Shane Kelly, Teruhisa Koshikawa, Arthur-C\'esar Le Bras, Zhouhang Mao, Akhil Mathew, Peter Scholze, Zijian Yao, and Yifei Zhao for helpful discussions and comments.
Also, the author would like to thank Bhargav Bhatt for explaining the second proof of Theorem \ref{Theorem:descent for perfect complex for derived quotients} given in Section \ref{Subsection:The second proof}.
Finally, the author would like to thank the referee for carefully reading the manuscript and for constructive comments.
The work of the author was supported by the European Research Council (ERC) under the European Union's Horizon 2020 research and innovation programme (grant agreement No.\ 851146).


\begin{thebibliography}{99}

\bibitem{AL}
    Ansch\"utz, J., Le Bras, A.-C.,
    \textit{Prismatic Dieudonn\'e theory},
    preprint, 2019, \texttt{arXiv:1907.10525}.

\bibitem{Berthelot}
    Berthelot, P.,
    \textit{Th\'eorie de Dieudonn\'e sur un anneau de valuation parfait},
    Ann.\ Sci.\ Ecole Norm.\ Sup.\ (4) \textbf{13} (1980), no.\ 2, 225--268.

\bibitem{BBM}
  Berthelot, P., Breen, L., Messing, W.,
  \textit{Th\'eorie de Dieudonn\'e cristalline.\ II},
  Lecture Notes in Mathematics, \textbf{930}. Springer-Verlag, Berlin, 1982.


\bibitem{BM}
    Bhatt, B., Mathew, A.,
    \textit{The $\arc$-topology},
    Duke Math.\ J.\ \textbf{170} (2021), no.\ 9, 1899--1988.

\bibitem{BMS}
  Bhatt, B., Morrow, M., Scholze, P.,
  \textit{Integral $p$-adic Hodge theory},
  Publ.\ Math.\ Inst.\ Hautes \'Etudes Sci.\
  \textbf{128} (2018), 219--397.
  
\bibitem{BS}
    Bhatt, B., Scholze, P.,
    \textit{Projectivity of the Witt vector affine Grassmannian},
    Invent.\ Math.\ \textbf{209} (2017), no.\ 2, 329--423.

\bibitem{CS}
    \v{C}esnavi\v{c}ius, K., Scholze, P.,
    \textit{Purity for flat cohomology},
    preprint, 2019, \texttt{arXiv:1912.10932}.

\bibitem{Cisinski}
Cisinski, D-C.,
\textit{Higher categories and homotopical algebra},
    Cambridge Studies in Advanced Mathematics, \textbf{180}, Cambridge University Press, Cambridge, 2019. 


\bibitem{deJong}
    de Jong, A.\ J.,
    \textit{Finite locally free group schemes in characteristic p and Dieudonn\'e modules},
    Invent.\ Math.\ \textbf{114} (1993), no.\ 1, 89--137.

\bibitem{Henkel}
    Henkel, T.,
    \textit{Classification of $\mathrm{BT}_n$-groups over perfectoid rings},
    PhD thesis, Technische Universit\"at Darmstadt, 2020.

\bibitem{Lau}
    Lau, E.,
    \textit{Smoothness of the truncated display functor},
    J.\ Amer.\ Math.\ Soc.\ \textbf{26} (2013), no.\ 1, 129--165.

\bibitem{Lau2018}
    Lau, E.,
    \textit{Dieudonn\'e theory over semiperfect rings and perfectoid rings},
    Compos.\ Math.\ \textbf{154} (2018), no.\ 9, 1974--2004.

  
\bibitem{HTT}
    Lurie, J.,
    \textit{Higher topos theory},
    Annals of Mathematics Studies, \textbf{170}, Princeton University Press, Princeton, NJ, 2009.

\bibitem{DAGVIII}
    Lurie, J.,
    \textit{Derived Algebraic Geometry VIII: Quasi-Coherent Sheaves and Tannaka Duality Theorems}, 2011, \url{https://www.math.ias.edu/~lurie/papers/DAG-VIII.pdf}.
    

\bibitem{HA}
    Lurie, J.,
    \textit{Higher Algebra}, \url{https://www.math.ias.edu/~lurie/papers/HA.pdf}.

\bibitem{SAG}
    Lurie, J.,
    \textit{Spectral algebraic geometry},\\ \url{https://www.math.ias.edu/~lurie/papers/SAG-rootfile.pdf}.

\bibitem{Mathew16}
    Mathew, A.,
    \textit{The Galois group of a stable homotopy theory}, Adv.\ Math.\ \textbf{291} (2016), 403--541.

\bibitem{Rydh2010}
    Rydh, D.,
    \textit{Submersions and effective descent of \'etale morphisms},
    Bull.\ Soc.\ Math.\ France \textbf{138} (2010), no.\ 2, 181--230. 

\bibitem{Scholze-Weinstein2013}
    Scholze, P., Weinstein, J.,
    \textit{Moduli of p-divisible groups},
    Camb.\ J.\ Math.\ \textbf{1} (2013), no.\ 2, 145--237.

\bibitem{Scholze-Weinstein}
  Scholze, P., Weinstein, J.,
  \textit{Berkeley lectures on $p$-adic geometry},
  Annals of Mathematics Studies, \textbf{389}, Princeton University Press, Princeton, NJ, 2020.



\bibitem{SP}
  The Stacks Project Authors,
  \textit{Stacks Project}, 2018,
  \texttt{https://stacks.math.columbia.edu}.
  
\end{thebibliography}
\end{document}